\documentclass[10pt]{article}

\usepackage[margin=1in]{geometry} 
\usepackage{graphicx} 
\usepackage{amsmath} 
\usepackage{amsfonts} 
\usepackage{amsthm} 
\usepackage{amssymb} 
\usepackage{enumerate}
\usepackage{algorithm}
\usepackage[noend]{algpseudocode}
\usepackage{authblk}

\newtheorem{thm}{Theorem}[section]
\newtheorem{lem}[thm]{Lemma}
\newtheorem{prop}[thm]{Proposition}
\newtheorem{cor}[thm]{Corollary}

\newtheorem{example}{Example}

\newcommand{\RR}{\mathbb{R}} 
\newcommand{\ZZ}{\mathbb{Z}} 

\newcommand{\RNum}[1]{({\uppercase\expandafter{\romannumeral #1\relax}})}
\newcommand{\norm}[1]{\left\lVert#1\right\rVert}
\newcommand{\tr}[1]{{\rm trace \,}}

\begin{document}
\title {A Non-Degenerate Perturbation of the Assignment Polytope and its Application to Graph Matching}
\author{Kevin M. Byrnes
\thanks{E-mail:\texttt{kbyrnes2@jhu.edu}}
\thanks{Submitted to Discrete Mathematics, Algorithms and Applications}}
\affil{T. Rowe Price, 100 E. Pratt St., Baltimore MD 21202}
\maketitle

\begin{abstract}
We consider maximizing a continuous convex function over the assignment polytope.  Such problems arise in Graph Matching (the optimization version of Graph Isomorphism) and Quadratic Assignment problems.  In the typical case of maximizing a convex function over a polytope the problem can be solved by using a simplicial algorithm such as Tuy's method or the Falk-Hoffman method, but these algorithms require that the underlying polytope be non-degenerate, which is not the case for the assignment polytope.  In this note we show how a simple perturbation scheme can be used to create a ``surrogate problem'' that is both non-degenerate and combinatorially equivalent to the original problem.  We further provide an explicit construction of a surrogate problem that is non-degenerate and combinatorially equivalent to the Graph Matching problem, when the latter is posed as a convex maximization problem.  By constructing a surrogate problem that is known a priori to be non-degenerate and combinatorially equivalent to Graph Matching we resolve an ``open issue'' of solving Graph Matching via convex maximization first raised by Maciel. 
\end{abstract}

\section {Introduction}
In this note we consider maximizing a convex function over the assignment polytope.  Such problems frequently arise in combinatorial optimization, for example the Graph Matching Problem (GM) \footnote{Here by Graph Matching we refer to the optimization version of the Graph Isomorphism problem, and not the problem of finding a maximum cardinality subset of edges such that no two edges in the subset share a vertex.}\cite{GM} and the Quadratic Assignment Problem (QAP) \cite{Sherali}.  Typically convex maximization problems can be solved by first observing that a maximizer of the objective function must lie at a vertex of the feasible polytope, and then using a simplicial algorithm such as Tuy's method \cite{Tuy}, or the Falk-Hoffman method \cite{FalkHoffman} to identify an optimal vertex.

One difficulty that arises in applying these algorithms is that they assume that the feasible polytope is non-degenerate, which is frequently not the case in combinatorial problems.  However the feasible region of a combinatorial problem can sometimes be perturbed so as to be non-degenerate.  This is the case with the assignment polytope, as was first shown by Orden \cite{Orden} and later researched by Liu \cite{Liu}. 

This suggests the strategy of perturbing the feasible region of the combinatorial problem (i.e. the assignment polytope) so as to become non-degenerate and then solving the resulting ``surrogate problem''.  Here we identify the range of values for a perturbation such that the optimal basis for the surrogate problem is also an optimal basis for the original problem (i.e. the problems are \emph{equivalent}).  This solves a practical problem of establishing a priori when the surrogate problem is equivalent without having to make the perturbation so small so as to invite numerical precision errors or to vastly increase the problem encoding size.

Finally, we demonstrate how our findings can be applied to the Graph Matching problem by first posing it as a convex maximization problem and then explicitly constructing an equivalent surrogate problem.  Although the optimality properties of the assignment polytope under a perturbation have been extensively studied in the case of a linear objective function (see \cite{Orlin}), we believe that the extension of this approach to the non-linear case, and in particular to GM, is novel.

\section {A Non-Degenerate Perturbation and the Optimal Basis Theorem}
\subsection{The Birkhoff Polytope and a Non-Degenerate Perturbation}
We consider the following optimization problem
\begin{equation*}
\begin{array}{llll}
& \text{maximize} & f(x) &\\
\RNum{1} &\text{subject to} &\displaystyle\sum\limits_{j=1}^{n} x_{ij} = 1, &i=1,\ldots,n\\
& &\displaystyle\sum\limits_{i=1}^{n} x_{ij} = 1&j=1,\ldots,n\\
& & x_{ij}\ge 0 & i=1,\ldots,n,\text{ } j=1,\ldots,n
\end{array}
\end{equation*}
where $f:\RR^{n\times n} \rightarrow \RR$ is convex and continuous.  Problem \RNum{1} generalizes QAP (canonically a \emph{minimization} problem) with negative coefficients and is therefore NP-hard \cite{GJ}.  However, we can still solve \RNum{1} using a simplicial algorithm such as Tuy's method \cite{Tuy} or the Falk-Hoffman method \cite{FalkHoffman}.  Unfortunately these algorithms assume that the feasible polytope is non-degenerate, which is not the case in \RNum{1}.  There are variants of these algorithms which do handle degeneracy, such as Meyer's subdivision variant of Tuy's method \cite{Meyer}, or Hoffman's generalization of the Falk-Hoffman method \cite{Hoffman}, but they either require introducing a branching framework (eliminating much of the simplicity of Tuy's and Falk and Hoffman's methods), or lose the guarantee of finite convergence.  Here we consider a different approach, introducing a non-degenerate surrogate problem that is equivalent to the original problem.

To construct the surrogate problem we begin with the following two results, to which we will often refer.

\begin{thm} [Birkhoff-von Neumann Theorem]
\label{BirkhoffPf}
The extreme points of the polytope $P$ defined by the linear system:
\begin{equation*}
\begin{array}{lll}
\sum_{j=1}^n x_{ij}=1 \text{ for } i=1,\ldots,n,
&\sum_{i=1}^n x_{ij}= 1 \text{ for } j=1,\ldots,n,
&x_{ij}\ge0
\end{array}
\end{equation*}
are precisely the set of $n\times n$ permutation matrices.
\end{thm}
\begin{proof}
See \cite{Birkhoff}, \cite{vonNeumann}.
\end{proof}

\begin{lem} [Liu, 2013]
\label{liuProof}
Let $t\in (0,\frac{1}{n})$, then the polytope defined by:
\begin{equation*}
\begin{array}{lll}
\sum_{j=1}^n x_{ij}=1-t \text{ for } i=1,\ldots,n,
&\sum_{i=1}^n x_{ij}= 
\begin{cases}
1, & \text{if } j=1,\ldots,n-1\\
1-nt, &\text{if } j=n
\end{cases},
&x_{ij}\ge0
\end{array}
\end{equation*}
is non-degenerate.
\end{lem}
\begin{proof}
See \cite{Liu} Lemma 4.4.
\end{proof}

Note that the polytope described in Thm. \ref{BirkhoffPf} is frequently referred to as the assignment polytope, or the Birkhoff polytope (we adopt the latter convention).  The properties of these polytopes were extensively studied by Brualdi and Gibson in an influential series of papers beginning with \cite{BG1}.  

Prior to that, Order \cite{Orden} studied these polytopes from an optimization perspective, in particular how to resolve degeneracy when solving a transportation problem.  He introduced a perturbation similar to the one in Lemma \ref{liuProof}, but with ``$1-t$'' and ``$1-nt$'' terms replaced by ``$1+t$''  and ``$1+nt$'' respectively.  This note was motivated by \cite{Liu} and assumes a perturbation as given above, but in fact our analysis holds for Orden's perturbation as well with minor adjustments.

The proposed form for the surrogate problem is given in \RNum{2}.  The final equality constraint is omitted as it is linearly dependent.
\begin{equation*}
\begin{array}{llll}
& \text{maximize} & f(x) &\\
\RNum{2} &\text{subject to} &\displaystyle\sum\limits_{j=1}^{n} x_{ij} = 1-t, &i=1,\ldots,n\\
& &\displaystyle\sum\limits_{i=1}^{n} x_{ij} = 1&j=1,\ldots,n-1\\
& & x_{ij}\ge 0 & i=1,\ldots,n,\text{ } j=1,\ldots,n
\end{array}
\end{equation*}

Clearly for some sufficiently small $t>0$ \RNum{2} is non-degenerate and has an optimal basis that is also an optimal basis of \RNum{1} (as the optimal solution is a continuous function of the right hand side of the constraints), and thus is equivalent to \RNum{1}.  At this point one might (rightfully) ask whether any value of $t$ for which \RNum{2} is non-degenerate would also result in a problem that is equivalent to \RNum{1}.  As the following numeric example shows, this is not necessarily the case.
\begin{example}
\label{exPerturbedDiffOpt}
Consider the problem:
\begin{equation*}
\begin{array}{ll}
\text{maximize} & f(x) = .0006(x_{12}+x_{21}) + 1.001 x_{11}^2 + x_{12}^2 + x_{21}^2 + 1.001 x_{22}^2 \\
\text{subject to} &x_{11}+x_{12}=1-t \\
& x_{21}+x_{22}=1-t \\
& x_{11}+x_{21}=1 \\
& x_{ij}\ge 0 \forall \text{ }i,j
\end{array}
\end{equation*}
When $t=\frac{1}{2}-.0005$ the unique optimal basic feasible solution is $x_{11}=.4995, x_{12}=.001, x_{21}=.5005,x_{22}=0$.  Clearly $x_{22}$ is non-basic here.  But when $t=0$ (i.e. the problem is unperturbed) the unique optimal basic feasible solution is $x_{11}=1, x_{12}=0, x_{21}=0, x_{22} =1$, which has $x_{22}$ basic\footnote{To verify that the perturbed problem has a unique optimal basic feasible solution it is important not to truncate the values of $f$.}.
\end{example}

Therefore our goal is to identify a range of values for $t$ so that it is known \emph{a priori} if \RNum{2} is equivalent to \RNum{1}.  Later we shall see that this resolves a practical problem in solving GM via concave maximization first raised in \cite{MacielThesis}.  

Before proceeding, we note that one might also ask whether existing sensitivity analysis techniques for linear programming, i.e. the 100\% rule of Bradley et al. \cite{Bradley} or the tolerance approach of Wendell \cite{Wendell} could easily be adapted to find the range of values for $t$.  The answer is ``no'' for two reasons:  first (and most obviously) the problem at hand is \emph{nonlinear}, and second these techniques consider when the current basis remains optimal after a perturbation.  As the following example shows, it is possible (due to degeneracy) for an optimal basis of \RNum{1} to be infeasible after perturbing \RNum{1} to \RNum{2}, even though \RNum{2} remains equivalent (in the way we defined this term earlier) to \RNum{1}.  This is because one, but not all, of the many optimal bases associated with the optimal vertex in \RNum{1} remains an optimal basis for \RNum{2}.

\begin{example}
\label{exPerturbedInfeas} 
Consider the problem:
\begin{equation*}
\begin{array}{lll}
\text{maximize} & f(x) =1.01 \sum_{i=1}^3 x_{ii}^2 + \sum_{i=1}^3 \sum_{j\neq i} x_{ij}^2 &\\
\text{subject to} &\sum_{j=1}^3 x_{ij}=1-t & i=1,\ldots,3\\
&\sum_{i=1}^3 x_{ij}=1 & j=1,2\\
& x_{ij}\ge 0 &\forall \text{ }i,j
\end{array}
\end{equation*}
When $t=0$ (the problem is unperturbed) there is an optimal solution with basic variables $x_{11}, x_{22}, x_{23}, x_{31}, x_{33}$.  This corresponds to the unique optimal \emph{vertex} $x^*$.  But the same basis is infeasible when the problem is perturbed with $t=.001$.  However, the solution whose basic variables are $x_{11}, x_{22}, x_{31}, x_{32}, x_{33}$ is optimal both for $t=0$ (where this basic solution also corresponds to $x^*$) and $t=.001$.  In fact, this is the unique optimal basic feasible solution for $t=.001$, thus the problems are equivalent but not all optimal bases are feasible in the perturbed problem.
\end{example}

\subsection{The Optimal Basis Theorem}
To begin with, observe that the constraints used in defining $P$ (the feasible region of \RNum{1}) are exactly the constraints used in what Dantzig and Thapa \cite{dantzigThapa} call the \emph{classical transportation problem}.

\begin{prop}
Let $Ax=b, x\ge 0$ be the constraints for the classical transportation problem. Then $A$ is a totally unimodular matrix.
\end{prop}
\begin{proof}
See \cite{dantzigThapa}. 
\end{proof}

Obviously if $A$ is totally unimodular then so is any subset of the rows of $A$ (e.g. the constraints of \RNum{1} with the final, linearly dependent, constraint omitted).

\begin{prop}
\label{infeasNegative}
Consider the polytope defined by $\{x|Ax=b, x\ge 0\}$ where $A$ is totally unimodular and of full row rank and $b$ is integer. Let $\mathcal{B}=\{B_{(1)},\ldots ,B_{(m)}\}$ be a basis defined from $A$ (i.e. $\mathcal{B}=\{A_{\pi(1)},\ldots,A_{\pi(m)}\}$), and let $\hat{x}$ denote the attendant basic solution. If $\hat{x}$ is not feasible, then $\hat{x}_j \le -1$ for some $j\in \{1,\ldots,n\}$.
\end{prop}
\begin{proof}
By Cramer's Rule, $\hat{x}_j\in \ZZ$ for $j=1,\ldots, n$ since $A$ is totally unimodular and $b\in \ZZ$. The result follows immediately.
\end{proof}

The total unimodularity of the constraint matrix allows us to bound the absolute change in the value of each component of a basic solution when the right hand side is perturbed.  In view of Prop. \ref{infeasNegative} this means that if the perturbation is small enough, then a basis that yields a feasible solution to the unperturbed problem will also yield a feasible solution to the perturbed problem.

\begin{prop}
\label{basicSolnDist}
Let $A\in \RR^{m\times n}$ be totally unimodular and of full row rank, and let $b\in \ZZ^m$. Let $\gamma \in \RR^m$ such that $|\gamma_i|<\frac{\Gamma}{m} \text{ for }i=1,\ldots, m$, and let $b'=b-\gamma$. Let $\mathcal{B}=\{B_{(1)},\ldots,B_{(m)}\}=\{A_{\pi(1)},\ldots,A_{\pi(m)}\}$ be a basis defined from $A$, and let $\hat{x}$ and $x'$ be the basic solutions defined by 
\begin{equation*}
\begin{array}{llcll}
\hat{x}_j=
&\begin{cases}
(B^{-1}b)_i, & \text{if } j=\pi_{(i)}\\
0, &\text{otherwise }
\end{cases}
&
&x'_j=
&\begin{cases}
(B^{-1}b')_i, &\text{if }j=\pi_{(i)}\\
0,&\text{otherwise }
\end{cases}
\end{array}
\end{equation*}
then $|\hat{x}_j-x'_j|< \Gamma$.
\end{prop}
\begin{proof}
Let $j=\pi_{(i)}$, then $x'_j=\frac{\det(\tilde{B})}{\det(B)}$, where $\tilde{B}=[A_{\pi(1)},\ldots,A_{\pi(i-1)},b',A_{\pi(i+1)},\ldots,A_{\pi(m)}]$. Using cofactor expansion,this yields $x'_j=\frac{\sum_{k=1}^m (-1)^{k+i}b'_k \det(\tilde{B}_{-k-i})}{\det(B)}$ where $\tilde{B}_{-k-i}$ is the submatrix of $\tilde{B}$ with row $k$ and column $i$ removed.  We compute $\hat{x}_j$ similarly, but with $b_k$ in place of $b'_k$. 

So $|\hat{x}_j-x'_j|=\left| \hat{x}_j-\frac{\sum_{k=1}^m (-1)^{k+i} (b_k-\gamma_k) \det(\tilde{B}_{-k-i})}{\det(B)}\right|$
=$\left| \frac{\sum_{k=1}^m (-1)^{k+i}\gamma_k \det(\tilde{B}_{-k-i})}{det(B)}\right| \le \frac{1}{|\det(B)|}\sum_{k=1}^m |\gamma_k| < \Gamma$
since $\tilde{B}_{-k-i}$ is totally unimodular as it is a submatrix of $A$ (note that $\det (B)\neq 0$ as $\mathcal{B}$ is a basis).
\end{proof}

\begin{cor}
\label{normCorr}
Let $A, b, b', \mathcal{B}, \gamma, \hat{x}, \text{and } x'$ be as defined in Prop. \ref{basicSolnDist}. Then $\norm{\hat{x}-x'}^2=\sum_j (\hat{x}_j-x'_j)^2< m \Gamma ^2$.
\end{cor}

If the right hand side defining the Birkhoff polytope is perturbed as in Lem. \ref{liuProof}, then a basis that defines a point that is feasible for the unperturbed polytope might correspond to a point that is infeasible for the perturbed polytope.  See, for example, Ex. \ref{exPerturbedInfeas}.  However, the following result shows that for any basic feasible solution $\hat{x}$ to the original problem, there is a feasible point $x’$ of the perturbed problem that is ``close'' to $\hat{x}$ (though it may correspond to a different basis).

\begin{prop}
\label{perturbFeas}
Let $P$ denote the Birkhoff polytope (represented with the final, linearly dependent, equality constraint omitted), and let $P'$ denote the perturbed polytope (parameterized by $t$) represented as in \RNum{2}. Let $\mathcal{B}$ denote a feasible basis for $P$, with attendant basic feasible solution $\hat{x}$. Then $\exists x'\in P'$ with $|\hat{x}_{ij}-x'_{ij}|\le nt$ for all $i,j$. 
\end{prop}
\begin{proof}
By Thm. \ref{BirkhoffPf}, for each $i=1,\dots,n$ $\hat{x}_{ij}=1$ for exactly one $j$, and 0 otherwise. Similarly, for each $j=1,\ldots,n$, ${\hat{x}_{ij}=1}$ for exactly one $i$ and 0 otherwise. Let $i^*$ be such that ${\hat{x}_{i^*n}=1}$. For $i\neq i^*$ define $x'_{ij}=\begin{cases} 1-t, &\text{if } \hat{x}_{ij}=1\\ 0,&\text{otherwise}\end{cases}$ and for $i=i^*$ define $x'_{ij}=\begin{cases} 1-nt,&\text{if } j=n\\ t,&\text{otherwise}\end{cases}$. 
\\Then $x'\in P'$ and trivially $|\hat{x}_{ij}-x'_{ij}|\le nt \text{ }\forall i,j$.
\end{proof}

Using the preceding sensitivity analysis, we are prepared for the main result of this section, which is a bound on the value of $t$ so that an optimal solution to the perturbed problem is also an optimal solution to the original problem.  Note that for $x\in \RR^{n\times n}$ we define $\norm{x}$ consistently with Cor. \ref{normCorr}\footnote{I.e. $\norm{x}$ denotes $(\sum_{i=1}^n \sum_{j=1}^n (x_{ij})^2)^\frac{1}{2}$.}.

\begin{thm}[Optimal Basis Theorem]
\label{thm:optBasisThm}
Let $P$ denote the Birkhoff polytope(represented with the final, linearly dependent, equality constraint omitted), and let $P'$ denote the perturbed polytope (parameterized by $t$) represented as in \RNum{2}. Let $t<\frac{\delta}{n^(2n-1)}$ where $\delta \in (0,1)$ satisfies $|f(x)-f(y)|<\frac{1}{2}$ if $x$ is a vertex of $P$ and $\norm{x-y}<\delta$. Furthermore, assume that $f$ is continuous, quasi-convex and integer valued\footnote{This is done without loss of generality as long as $f$ is rational-valued on $vert(P)$.} on $vert(P)$. Let $\mathcal{B}^*$ be a basis such that $f$ is maximized (over $P'$) at the attendant basic feasible solution $x^*$. Then $\mathcal{B}^*$ is also an optimal basis for $P$, i.e. $f(x)$ is maximized over $P$ at $\hat{x}$, where $\hat{x}$ is the basic feasible solution corresponding to $\mathcal{B}^*$.
\end{thm}
\begin{proof}
Since $f$ is quasi-convex it suffices to show that $f(\hat{x})=max_{x\in vert(P)} f(x)$. Since $\delta<1$ by Prop. \ref{basicSolnDist} we have $|x^*_{ij}-\hat{x}_{ij}|<1 \text{ }\forall i,j$\footnote{Here we take $m=2n-1, \Gamma=\frac{\delta}{n}$.}. Since $x^*_{ij}\ge 0 \forall i,j$ this implies by Prop. \ref{infeasNegative} that $\hat{x}_{ij}>-1 \text{ }\forall i,j$ and hence (since $\hat{x}$ is a basic solution) that $\hat{x}\in P$ ($\hat{x}$ is a basic feasible solution).
Now let $x'$ be an arbitrary basic feasible solution of $P$, then by Prop. \ref{perturbFeas} $\exists x''\in P'$ such that $|x'_{ij}-x''_{ij}|\le n\frac{\delta}{n(2n-1)}=\frac{\delta}{(2n-1)}$ $\forall i,j$. This implies $\norm{x'-x''}^2\le (2n-1)\frac{\delta^2}{(2n-1)^2}<\delta$. Thus $|f(x')-f(x'')|<\frac{1}{2}$.
Also, by Prop. \ref{basicSolnDist} and Cor. \ref{normCorr} we have $|f(x^*)-f(\hat{x})|<\frac{1}{2}$. Finally, by the optimality of $x^*$ (over $P'$) we have $f(x^*)\ge f(x'')$ so
\begin{equation*}
\begin{array}{lcl}
f(x')-f(\hat{x})&<&f(x'')+\frac{1}{2}-(f(x^*)-\frac{1}{2})\\
&=&f(x'')-f(x^*)+1\\
&\le&1
\end{array}
\end{equation*}
Thus $f(x')-f(\hat{x})<1$. But since $f(x'),f(\hat{x})\in \ZZ$, this implies that $f(x')\le f(\hat{x})$. Thus $\mathcal{B}^*$ is an optimal basis for $P$.
\end{proof}

In practice, finding an exact solution of \RNum{1} (via solving \RNum{2}) may be too costly computationally.  Instead we may prefer to terminate an optimization algorithm for solving \RNum{2} before it converges in order to simply bound the optimality gap of a pre-existing feasible solution to \RNum{1}.  The following corollary establishes that an upper bound on the optimal value of \RNum{2} can easily be rounded to an upper bound of the optimal value of \RNum{1}.  

\begin{cor}
\label{integerBound}
Let $P$ denote the Birkhoff polytope, and let $f(x)$ be a continuous convex function that is integer valued on the vertices of $P$.  And let $\delta\in (0,1)$ satisfy $|f(x)-f(y)|<\frac{1}{2}$ if $x$ is a vertex of $P$ and $\norm{x-y}<\delta$.  Let $P'$ denote the perturbed polytope of \RNum{2} with $t\in (0, \frac{\delta}{n^(2n-1)})$.  Then any integer upper bound of $max_{x\in vert(P')} f(x)$ is a valid upper bound of $max_{v\in vert(P)} f(x)$.
\end{cor}

Given any simplicial algorithm for maximizing a convex function over a polytope that outputs a successively tight sequence of upper bounds for the true optimal value, we can bound the optimality gap of a feasible solution of \RNum{1} as follows:  first transform \RNum{1} to \RNum{2}, run the simplicial algorithm for $k$ iterations to get an upper bound of the optimal value of \RNum{2}, finally round this upper bound up to the next largest integer, this is an upper bound for the optimal value of \RNum{1}.  Both Tuy's method \cite{Tuy} and the Falk-Hoffman method \cite{FalkHoffman} output such a sequence of successively tight upper bounds.  In particular, it is this property (in conjunction with Cor. \ref{integerBound}) that makes these algorithms attractive choices to empirically bound the optimality gap of a feasible solution to \RNum{1}.

\section{An Application to Graph Matching}
In the Graph Matching Problem (GM) one is given two simple, undirected, graphs of $n$ vertices $G_1 = (V,E^1)$ and $G_2 = (V,E^2)$ the objective is to relabel the vertices in $G_1$ according to some permutation $\sigma:\{1,\ldots,n\}\rightarrow \{1,\ldots,n\}$ such that \emph{after relabeling} the size of the symmetric difference of the edge sets $|E^1\triangle E^2|$ is minimized.  Currently there is no known polynomial time algorithm known for determining whether or not $G_1$ and $G_2$ are isomorphic, and so there is no known polynomial algorithm for solving the optimization problem GM in the general case\footnote{There are however, polynomial algorithms to solve GM in the case that $G_1$ and $G_2$ are trees or planar graphs.}.

Notice that there is a natural bijection between permutations of $\{1,\ldots,n\}$ and $n\times n$ permutation matrices, and that $|E^1\triangle E^2|$ is minimized (after applying the permutation $\sigma$) if and only if the permutation matrix $x=(x_{ij})$ corresponding to $\sigma$ minimizes the Froebenius norm $\norm{E^1x-xE^2}_{F}^2$ (which counts the number of adjacency disagreements induced by $\sigma$).  (Here $E^1$ and $E^2$, in a slight abuse of notation, represent the adjacency matrices of graphs $G_1$ and $G_2$ respectively.)  Thus (GM) can be formulated as the minimization of $\norm{E^1x-xE^2}_{F}^2$ over all $n\times n$ matrices $x$, subject to linear constraints and integrality of the decision variables to ensure that the solution corresponds to a permutation matrix\footnote{See \cite{GM} for a comprehensive survey of GM and the ``classical'' formulation of GM as an optimization problem.}.

However, as Lyzinski et al. \cite{LyzinskiPNAS} observe\footnote{Specifically, Lyzinksi et al. observed that minimizing $\norm{E^1x-xE^2}_{F}^2$ over all $n\times n$ permutation matrices $x$ is the same as minimizing $-\tr (((E^1x)^T E^2x)$.  And for any square matrices $U$ and $V$ of the same dimension, $U\bullet V := \tr ((U^T V)$.}, (GM) can also be formulated as the maximization of a (typically indefinite) quadratic function over a discrete set of points as follows.  
\begin{equation*}
\begin{array}{llll}
& \text{maximize} & E^1x\bullet xE^2 &\\
\RNum{3} &\text{subject to} &\displaystyle\sum\limits_{j=1}^{n} x_{ij} = 1, &i=1,\ldots,n\\
& &\displaystyle\sum\limits_{i=1}^{n} x_{ij} = 1, &j=1,\ldots,n\\
& & x_{ij}\in \{0,1\} & i=1,\ldots,n, j=1,\ldots,n
\end{array}
\end{equation*}
Here the decision variables are $x_{ij} =\begin{cases} 1 &\mbox{if } \sigma(i)=j\\ 0 &\mbox{otherwise} \end{cases}$. 

Note that $E^1x\bullet xE^2$ is a quadratic form in $x$ and that $E^1$ and $E^2$ are real matrices.  Therefore $E^1x\bullet xE^2=\mathbf x^TQ\mathbf x$ where $\mathbf x=\verb|vec|(x)$ and $Q$ is a real symmetric matrix.  This leads to the following reformulation of GM as the maximization of a continuous convex function over a polytope using an observation of \cite{Sherali}.

\begin{prop}
\label{convexMaxFormulation}
Let $\lambda=max_i\{|\lambda^i| | \lambda^i \text{is an eigenvalue of }Q\}$.  For any $\mu>\lambda$, $x^*$ solves \RNum{4} if and only if $x^*$ solves \RNum{3}.
\begin{equation*}
\begin{array}{llll}
& \text{maximize} & \mathbf x^T[Q+\mu I]\mathbf x &\\
\RNum{4} &\text{subject to} &\displaystyle\sum\limits_{j=1}^{n} x_{ij} = 1, &i=1,\ldots,n\\
& &\displaystyle\sum\limits_{i=1}^{n} x_{ij} = 1, &j=1,\ldots,n\\
& & x_{ij}\ge 0 & i=1,\ldots,n, j=1,\ldots,n
\end{array}
\end{equation*}
\end{prop}
\begin{proof}
Let $f(x)=\mathbf x^T[Q+\mu I] \mathbf x$ and let $P$ denote the feasible region of \RNum{4}.  Observe that $f$ is strictly convex and $P$ is a polytope, and so every maximizer of $f$ is a vertex of $P$.  Furthermore by Thm. \ref{BirkhoffPf} for any pair of vertices $x\text{ and } x'$ of $P$, both $x\text{ and } x'$ are binary-valued with $\sum_{i,j} x_{ij} = \sum_{i,j} x'_{i,j} = n$, and thus $\mu \mathbf x^TI \mathbf x = \mu \mathbf x'^TI \mathbf x'$.  Hence $x^*$ solves \RNum{4} implies ${x^*\in argmax_{x\in vert(P)} \mathbf x^TQ \mathbf x}$.  Since the feasible solutions of \RNum{3} are exactly the set of $n\times n$ permutation matrices ($=vert(P)$ by Thm. \ref{BirkhoffPf}) and since $E^1x\bullet xE^2 = \mathbf x^T Q\mathbf x$, this means that $x^*$ solves \RNum{3} as well.

Conversely, if $x^*$ solves \RNum{3}, then 
\begin{equation*}
\begin{array}{lll}
x^* &\in &argmax_{x\in vert(P)} \mathbf x^TQ \mathbf x\\
&= &argmax_{x\in vert(P)}\mathbf x^TQ \mathbf x + \mu n\\
& = &argmax_{x\in vert(P)}f(x)\\
&=&argmax_{x\in P}f(x)
\end{array}
\end{equation*}
by the convexity of $f$.  So $x^*$ solves \RNum{4}.
\end{proof}

Form \RNum{4} is a popular way of formulating (GM) (see \cite{MacielThesis}, \cite{MacielIEEE}, \cite{zLiu}) because the resulting problem is a convex maximization problem.  Although transforming (GM) into a type of problem that is NP-hard in general seems like an unusual tactic, it is expedient because convex maximization problems possess enough structure that there are specialized algorithms to solve them (such as the algorithms of Tuy and Falk and Hoffman mentioned previously), and at present there is no known efficient algorithm for (GM) in general.  Practically, the choice to treat (GM) as a convex maximization problem is justified by the good empirical performance of convex maximization algorithms on many classes of small and mid-sized problems \cite{FalkHoffman2}.  Of particular relevance are the numerical results in \cite{MacielThesis} \cite{MacielIEEE} which show that the convex maximization approach works well for large graph matching problems arising in the context of image correspondence.

As mentioned earlier, a potential difficulty arises in that many popular convex maximization algorithms assume that the feasible region is a non-degenerate polytope, which is not true for \RNum{4}.  Maciel was the first to raise this ``open issue'' in \cite{MacielThesis} but the solution therein seems only to apply when solving \RNum{4} with the Cabot-Francis method \cite{cabotFrancis}.  Although convex maximization algorithms can sometimes be adapted to handle degeneracy, doing so may compromise the good empirical performance which made convex maximization an appealing solution technique in the first place\footnote{Here we refer to the need to maintain a branching structure to modify Tuy's method for the degenerate case \cite{Meyer}, or the loss of finite convergence when adapting the method of Falk and Hoffman \cite{Hoffman}.}.

Alternately \RNum{4} could be perturbed so the feasible polytope is non-degenerate, in which case we would like an a priori guarantee that the perturbation is sufficiently small so that an optimal solution of the perturbed problem corresponds to an optimal solution of the original problem.  As Ex. \ref{exPerturbedDiffOpt} shows, this is not always obvious.  However, since \RNum{4} is of form \RNum{1}, it can be perturbed into an equivalent surrogate problem of form \RNum{2}.  An elementary linear algebra argument provides the bound on $\delta$ required by Thm. \ref{thm:optBasisThm}.

\begin{prop}
\label{deltaBd}
Let $f(x) = \mathbf x^T[Q+\mu I] \mathbf x$, where $Q$ and $\mu$ are defined as in Prop. \ref{convexMaxFormulation}, and let $P$ be the Birkhoff polytope. If $x$ is a vertex of $P$ and $\norm{x-y} < min(1,\frac{1}{4\mu(2\sqrt{n}+1)})$ then $|f(x)-f(y)|<\frac{1}{2}$.
\end{prop}
\begin{proof}
Let $\bar{Q}$ denote $[Q+\mu I]$, and observe that $\bar{Q}$ is a real symmetric matrix, and thus has an orthonormal eigenbasis. Also note that $\bar{Q}$ has its largest eigenvalue (in absolute value) bounded by $2\mu$. Let $\mathbf d = \mathbf y- \mathbf x$ (here \textbf{bold} font represents the vectorization operator as before), then $|f(y)-f(x)|\le 2|\mathbf x^T\bar{Q}\mathbf d| + 2\mu \norm{\mathbf d}^2$. By Cauchy-Schwartz $|\mathbf x^T\bar{Q}\mathbf d|\le \sqrt{n}2\mu \norm{\mathbf d}$ as $\norm{\mathbf x}^2 = n$ for every vertex of $P$. Assuming that $\norm{\mathbf d} \le 1$ $|f(y)-f(x)|$ has an upper bound of $\norm{\mathbf d}(2\mu (2\sqrt{n} +1)$.  Thus for $\norm{x-y}=\norm{\mathbf d}<min(1,\frac{1}{4\mu(2\sqrt{n}+1)})$ we have $|f(x)-f(y)|<\frac{1}{2}$ as desired.
\end{proof}

\section{Conclusions}
In this note we considered the problem of maximizing a continuous convex function over the assignment polytope.  We discussed how such problems could be solved by simplicial algorithms (such as Tuy's method \cite{Tuy} and the Falk-Hoffman method \cite{FalkHoffman}) by forming an equivalent (in the sense we defined the term earlier), non-degenerate, ``surrogate'' problem using Liu's perturbation \cite{Liu}. 

While it is clear that for some sufficiently small perturbation the resulting surrogate problem is both non-degenerate and equivalent there remains a practical issue of not making the perturbation so small so as to invite numerical precision errors or to vastly increase the encoding size of the problem.  Our primary contribution (Thm. \ref{thm:optBasisThm}) was to determine a bound on the size of the perturbation such that all perturbations smaller than this are guaranteed to yield surrogate problems that are both non-degenerate and equivalent.  

We further provide an explicit construction of a surrogate problem that is non-degenerate and combinatorially equivalent to the Graph Matching problem, when the latter is posed as a convex maximization problem (such as in \cite{MacielThesis}, \cite{MacielIEEE}, \cite{zLiu}).  By constructing a surrogate problem that is known a priori to be non-degenerate and combinatorially equivalent to Graph Matching we resolve an ``open issue'' of solving Graph Matching via convex maximization first raised by Maciel \cite{MacielThesis}. In addition, our Cor. \ref{integerBound} shows how any algorithm that produces a successively tight sequence of upper bounds to the optimal value of the surrogate problem (e.g. \cite{Tuy}, \cite{FalkHoffman}, \cite{FalkHoffman2}, \cite{Hoffman}, \cite{Meyer}) can be run for a prescribed number of iterations to generate an empirical bound on the optimality gap for a known feasible solution to (GM).  

Some obvious questions remain unanswered, such as whether our bound on the size of the permutation is tight, and if so, when.  In addition, although our result applies specifically to the assignment polytope, the non-degenerate perturbations of Liu \cite{Liu} and Orden \cite{Orden} apply to the more general class of transportation polytopes.  It seems likely that a modification of our approach could be used to extend our result to all transportation polytopes, and we plan to pursue this line of investigation in a future paper.\\
\\
\textbf{Acknowledgements:}
The author thanks Stephen R. Chestnut for many helpful comments which greatly improved the quality of this work.  The author also wishes to thank Eric Harley, Donniell Fishkind, and Vince Lyzinksi for several stimulating discussions which lead to the development of the results herein, in particular with application to Graph Matching.

\bibliographystyle{amsplain}
\bibliography{Transportation_Perturbation_Bib_12_7_2014}
\end{document}